\documentclass[11pt,a4paper]{article}
\usepackage{amssymb,amsmath,mathrsfs,enumerate}
\numberwithin{equation}{section}

\usepackage[colorlinks=true, pdfstartview=FitV, linkcolor=blue, citecolor=blue, urlcolor=blue,pagebackref=false]{hyperref}

\usepackage{tikz}
\usepackage{float}
\usepackage[font=footnotesize]{caption}

\usepackage{color} 
\usepackage{times,theorem,latexsym,color,comment}

\newcommand{\BOX}{\ensuremath\Box}

\newtheorem{theorem}{Theorem }[section]

{\theorembodyfont{\rmfamily}}
{\theorembodyfont{\rmfamily}}
{\theorembodyfont{\rmfamily}}
\newtheorem{lemma}[theorem]{Lemma}
\newtheorem{proposition}[theorem]{Proposition}
{\theorembodyfont{\rmfamily}\newtheorem{remark}[theorem]{Remark}}
{\theorembodyfont{\rmfamily}}

\newcommand{\Z}{\mathbb{Z}}
\newcommand{\R}{\mathbb{R}}

\newcommand{\dd}{\,{\rm d}}

\def\XXint#1#2#3{{\setbox0=\hbox{$#1{#2#3}{\int}$}
		\vcenter{\hbox{$#2#3$}}\kern-.5\wd0}}

\DeclareMathOperator*{\esssup}{ess\,sup}

\newenvironment{proof}{{\vskip\baselineskip\noindent\textbf{Proof:}}}%
{\hspace*{.1pt}\hspace*{\fill}\BOX\vskip\baselineskip}

\newenvironment{proofx}[1]%
{\vskip\baselineskip\noindent\textbf{Proof of {#1}:}}%
{\hspace*{.1pt}\hspace*{\fill}\BOX\vskip\baselineskip}
{\vskip\baselineskip\noindent\textbf{Proof of Theorem \protect\ref{#1}:}}%
{\hspace*{.1pt}\hspace*{\fill}\BOX\vskip\baselineskip}
{\vskip\baselineskip\noindent\textbf{Proof of Theorems \protect\ref{#1} --
		\protect\ref{#2}:}}%
{\hspace*{.1pt}\hspace*{\fill}\BOX\vskip\baselineskip}

\begin{document}

\title{Existence of planar non-symmetric stationary flows \\ with large flux in an exterior disk}

\author{Mitsuo Higaki}
\date{}

\maketitle

\noindent {\bf Abstract.}\ This paper is concerned with the two-dimensional stationary Navier-Stokes system in the domain exterior to the unit disk. The existence of solutions with critical decay $O(|x|^{-1})$ is established around some explicit flows with large flux. The solutions are obtained for non-symmetric external forces, and moreover, are unique in a certain class.

\medskip


\medskip

\section{Introduction}\label{sec.intro.}

We consider the two-dimensional stationary Navier-Stokes system in the exterior disk $\Omega=\{x=(x_1,x_2)\in\R^2~|~|x|>1\}$:
\begin{equation}\tag{NS}\label{eq.NS.intro}
\left\{
\begin{array}{ll}
-\Delta u + \nabla p = -u\cdot\nabla u + f&\mbox{in}\ \Omega \\
{\rm div}\,u = 0&\mbox{in}\ \Omega \\
u= \alpha x^{\bot} - \gamma x &\mbox{on}\ \partial\Omega \\
u(x)\to0&\mbox{as}\ |x|\to\infty.
\end{array}\right.
\end{equation}
The unknown functions $u=(u_1(x),u_2(x))$ and $p=p(x)$ are respectively the velocity field and the pressure field. The function $f=(f_1(x),f_2(x))$ is a given external force in $L^2(\Omega)^2$. We assume that both $\alpha$ and $\gamma$ are constants. The vector $x^\bot$ refers to $(-x_2,x_1)$. The system \eqref{eq.NS.intro} describes the motion of a viscous incompressible fluid around the disk rotating at angular velocity $\alpha$ and whose surface subjects to a normal suction velocity $-\gamma x$.

The existence and uniqueness theories for the problem \eqref{eq.NS.intro} are generally open. As for the existence, there are two fundamental difficulties. The first one is lack of certain embeddings. Let $f$ be smooth and compactly supported in $\Omega$ for brevity. Under smallness conditions, adapting the proof by Leray \cite{Leray1933} or relying on Fijita \cite{Fujita1961}, we can actually find weak solutions of \eqref{eq.NS.intro} having a finite Dirichlet integral, called $D$-solutions. Nevertheless, the bound $\|\nabla u\|_{L^2}<\infty$ itself cannot verify the condition at spatial infinity in \eqref{eq.NS.intro} in two-dimensional unbounded domains; see Korobkov, Pileckas and Russo \cite{KorobkovPileckasRusso2019, KorobkovPileckasRusso2020, KorobkovPileckasRusso2021} for recent progress on this topic. The other one is the logarithmic growth in the Green function of the exterior Stokes system, which is the source of the famous Stokes paradox; see \cite{ChangFinn1961, Galdi04, Galdibook2011, KozonoSohr1992} for descriptions. The uniqueness of solutions will be discussed in Remark \ref{rem.thm.main} (\ref{item3.rem.thm.main}) below.

These issues illustrate a delicate aspect concerning the zero condition at infinity in \eqref{eq.NS.intro}. Most existing results treating \eqref{eq.NS.intro} assume both smallness and symmetry on the data. The latter is useful for making quantities decay in space by cancellation. The reader is referred to Galdi \cite{Galdi04, Galdibook2011}, Russo \cite{Russo09}, Yamazaki \cite{Yamazaki11, Yamazaki16, Yamazaki18} and Pileckas and Russo \cite{PileckasRusso12} for the existence of solutions, and to Nakatsuka \cite{Nakatsuka15} and \cite{Yamazaki18} for the uniqueness. Relevant numerical simulations are carried out in Guillod and Wittwer \cite{GuillodWittwer2015}. We note that, if another condition were imposed at infinity, the situation would be quite different. Indeed, the recent work by Korobkov and Ren \cite{KorobkovRen2021} proves the uniqueness in the class of $D$-solutions for plane exterior Navier-Stokes flows converging to a small but non-zero constant vector field at infinity.

In this paper, we examine the problem \eqref{eq.NS.intro} from a different angle. Note that, when the external force $f$ is trivial, there is an explicit solution $(\alpha U - \gamma W, \nabla P_{\alpha,\gamma})$ given by 
\begin{align}\label{def.U.W.P}
\begin{split}
U(x) = \frac{x^\bot}{|x|^2}, \qquad
W(x) = \frac{x}{|x|^2}, \qquad 
\nabla P_{\alpha,\gamma}(x) = -\nabla\Big(\frac{|\alpha U(x) - \gamma W(x)|^2}{2}\Big), 
\end{split}
\end{align}
which is invariant under the scaling of the Navier-Stokes equations. A (non-trivial) solution in this class is called scale-critical and it represents the balance between the nonlinear and linear parts of the equations. Given this nature, it is expected that the Navier-Stokes flows around a scale-critical flow differ quantitatively from those around the trivial flow. In fact, Hillairet and Wittwer \cite{HillairetWittwer13} consider the Navier-Stokes problem in an exterior disk by perturbating the system around $(\alpha U,\nabla P_{\alpha,0})$. A crucial observation is that the decay of solutions to the corresponding linearized system is improved when $|\alpha|$ is sufficiently large, more precisely, when $|\alpha|>\sqrt{48}$. This is possible because of the structure of the equation for vorticities. Then, based on iteration to the nonlinear problem with subcritialized nonlinearity, they show the existence of solutions in the form of $u(x) = \alpha U(x) + o(|x|^{-1})$ when $|x|\to\infty$. These solutions are driven by inhomogeneous boundary data, on which no symmetries are imposed thanks to the mechanism of the proof.

The result in \cite{HillairetWittwer13} can be read as the scale-critical flow $\alpha U$ producing a stabilizing effect in view of spatial decay. In this context, we here consider such an effect of $\alpha U - \gamma W$ and address the existence of solutions to the problem \eqref{eq.NS.intro}. Briefly, we will see that the stabilization is effective for any $\alpha$ if $\gamma>2$. As an application, we obtain the Navier-Stokes flows for non-symmetric external forces based on the perturbation. We also describe the asymptotics near spatial infinity in terms of the decay rate of external forces.

Let us introduce some notations to state the main result. For $s\ge0$, we define 
\begin{align}\label{def.L^infty_s}
\begin{split}
L^\infty_s(\Omega) &= \{ f\in L^\infty(\Omega)~|~ \|f\|_{L^\infty_s} <\infty \}, \qquad
\|f\|_{L^\infty_s} := \esssup_{x\in\Omega}\, |x|^s |f(x)|, 
\end{split}
\end{align}
which is a Banach space under the norm $\|\cdot\|_{L^\infty_s}$. Taking advantage of symmetry, we introduce the polar coordinates on $\Omega$ as 
\begin{align*}
&x_1 = r\cos \theta, \qquad
x_2 = r\sin \theta, 
\quad 
r= |x| \ge 1, \quad \theta\in [0,2\pi),\\
&
{\bf e}_r = \frac{x}{|x|}, \qquad 
{\bf e}_\theta = \frac{x^\bot }{|x|} = \partial_\theta {\bf e}_r. 
\end{align*}
For given vector field $v=(v_1,v_2)$ on $\Omega$ and $n\in\Z$, we set 
\begin{align*}
v = v_r(r,\theta) {\bf e}_r + v_\theta(r,\theta) {\bf e}_\theta, \qquad
v_r = v\cdot {\bf e}_r, \qquad
v_\theta = v\cdot {\bf e}_\theta 
\end{align*}
and denote by $\mathcal{P}_n$ the projection on the Fourier mode $n$: 
\begin{align}\label{def.P_n}
\begin{split}
\mathcal{P}_n v(r,\theta)
& = v_{r,n}(r) e^{i n \theta} {\bf e}_r + v_{\theta,n}(r) e^{i n \theta} {\bf e}_\theta, \\
v_{r,n} (r) 
& := \frac{1}{2\pi} \int_0^{2\pi} v_r(r \cos \theta, r\sin\theta) e^{-i n \theta} \dd \theta,\\
v_{\theta,n} (r) 
& := \frac{1}{2\pi} \int_0^{2\pi} v_\theta(r \cos \theta, r\sin\theta) e^{-i n \theta} \dd \theta.
\end{split}
\end{align}
Now the result is stated as follows.

%
\begin{theorem}\label{thm.main}
Let $\alpha\in\R$ and $\gamma>2$. Then, for any $2<\rho<3$ with $\rho\le\min\{\gamma,3\}$, there is a constant $\varepsilon=\varepsilon(\alpha,\gamma,\rho)>0$ such that if $f\in L^\infty_{2\rho-1}(\Omega)^2$ satisfies
\begin{align}\label{est1.thm.main}
\sum_{n\in\Z} \|\mathcal{P}_n f\|_{L^\infty_{2\rho-1}} \le \varepsilon, 
\end{align}
then there is a solution $(u,\nabla p)\in \big(\widehat{W}^{1,2}(\Omega) \cap W^{2,2}_{{\rm loc}}(\overline{\Omega}) \cap L^\infty_{1} (\Omega)\big)^2 \times L^2_{{\rm loc}} (\overline{\Omega})^2$ of \eqref{eq.NS.intro} unique in a suitable set of functions (see Section \ref{sec.nonlin.} for the precise definition). Moreover, when $|x|\rightarrow \infty$, the solution $u=u(x)$ behaves as  
\begin{align}\label{est2.thm.main} 
u(x) = \Big(\alpha \frac{x^{\bot}}{|x|^2} -\gamma \frac{x}{|x|^2}\Big) + O(|x|^{-\rho+1}), \qquad
{\rm rot}\,u(x) = O(|x|^{-\rho}). 
\end{align}
\end{theorem}
%

%
\begin{remark}\label{rem.thm.main}
\begin{enumerate}[(i)]
\item\label{item1.rem.thm.main} As far as the author knows, Theorem \ref{thm.main} is the first result obtaining the plane exterior Navier-Stokes flows zero at spatial infinity, non-symmetric and unique in a certain class. In \cite{HillairetWittwer13} dealing with the case $\gamma=0$, the uniqueness of solutions seems not to be provided. More to the point, the proof should not be easy since the existence of solutions in \cite{HillairetWittwer13} is verified by the intermediate value theorem. The inconvenience is closely related to the structure of the zero mode of the linearized system \eqref{eq.S.zeromode} below.

\item\label{item2.rem.thm.main} Contrary to \cite{HillairetWittwer13} where $|\alpha|>\sqrt{48}$ is assumed, we do not need any restrictions on $\alpha$ in Theorem \ref{thm.main}. Thus it is interpreted that the stabilization from $-\gamma W$ with $\gamma>2$ exceeds that from $\alpha U$, and hence, the theorem holds thanks to the presence of $\gamma$.

\item\label{item3.rem.thm.main} Theorem \ref{thm.main} does not state the uniqueness of $D$-solutions of \eqref{eq.NS.intro}, which is violated even when $f=0$ by the following Hamel solutions. For $\gamma>2$ and $t\in\R$, we set 
\begin{align*}
\begin{split}
A(x) &= \frac{1}{\gamma-2} (1 - |x|^{-\gamma+2}) U(x), \\
\nabla Q(x) &= - \nabla\Big(\frac{|\alpha U(x) - \gamma W(x) + t A(x)|^2}{2}\Big) \\
&\quad
-\nabla \bigg[t|x|^{-\gamma} \Big\{\frac{\alpha}{\gamma}+ \frac{t}{\gamma-2}\Big(\frac{1}{\gamma} - \frac{1}{2(\gamma-1)}|x|^{-\gamma+2}\Big)\Big\} \bigg]. 
\end{split}
\end{align*}
Then $\{(\alpha U - \gamma W + t A, \nabla Q)\}_{t\in\R}$ is a family of explicit $D$-solutions of \eqref{eq.NS.intro} with $f=0$. For more detailed discussion, the reader is referred to \cite[Section X\hspace{-.1em}I\hspace{-.1em}I.2]{Galdi04}.

\item\label{item4.rem.thm.main} The stability of the flows in Theorem \ref{thm.main} is an open problem. The global $L^2$-stability of $-\gamma W$ is proved by Guillod \cite{Guillod17} and the local $L^2$-stability of $\alpha U - \gamma W$ by Maekawa \cite{Maekawa17i, Maekawa17ii} without any symmetries on initial perturbations. However, these results essentially rely on the smallness of the coefficients and, therefore, cannot be adapted to the flows in Theorem \ref{thm.main}, even if the Hardy inequality applies to $u-(\alpha U-\gamma W)$. 
\end{enumerate}
\end{remark}
%

The ingredients of the proof of Theorem \ref{thm.main} are the following.

\noindent {\bf (I) Perturbation.} We will construct the solution $(u,\nabla p)$ of \eqref{eq.NS.intro} in the form of 
\begin{align}\label{intro.ansatz}
u = \alpha U - \gamma W + v, \qquad
\nabla p = \nabla\Big(-\frac{|u|^2}2 + q\Big). 
\end{align}
The pair $(v,\nabla q)$ is understood as the perturbation from $(\alpha U - \gamma W, \nabla P_{\alpha,\gamma})$ in response to the external force $f$. Inserting the ansatz \eqref{intro.ansatz} into \eqref{eq.NS.intro} and using the relation 
\begin{align}\label{eq.nonlin.rot}
u\cdot\nabla u = u^{\bot} {\rm rot}\,u + \nabla\Big(\frac{|u|^2}2\Big), \qquad 
{\rm rot}\,u := \partial_1 u_2 - \partial_2 u_1, 
\end{align}
as well as ${\rm rot}\,U={\rm rot}\,W=0$, we see that $(v,\nabla q)$ solves the perturbed system 
\begin{equation}\tag{$\widetilde{\mbox{NS}}$}\label{eq.NS.tilde.intro}
\left\{
\begin{array}{ll}
-\Delta v + (\alpha U - \gamma W)^\bot {\rm rot}\,v + \nabla q 
= -v^{\bot} {\rm rot}\,v + f &\mbox{in}\ \Omega \\
{\rm div}\,v =0&\mbox{in}\ \Omega \\
v=0&\mbox{on}\ \partial\Omega \\
v(x)\to0&\mbox{as}\ |x|\to\infty. 
\end{array}\right.
\end{equation}
Our next task is to prove the existence of solutions to the problem \eqref{eq.NS.tilde.intro}.

\noindent {\bf (I\hspace{-.1em}I) Linear analysis.} The linearized system of \eqref{eq.NS.tilde.intro} around $v=0$ is written as 
\begin{equation}\tag{$\widetilde{\mbox{S}}$}\label{eq.S.tilde.intro}
\left\{
\begin{array}{ll}
-\Delta v + (\alpha U - \gamma W)^\bot {\rm rot}\,v + \nabla q 
= f &\mbox{in}\ \Omega \\
{\rm div}\,v =0&\mbox{in}\ \Omega \\
v=0&\mbox{on}\ \partial\Omega \\
v(x)\to0&\mbox{as}\ |x|\to\infty. 
\end{array}\right.
\end{equation}
Since \eqref{eq.S.tilde.intro} is invariant under the action of $\mathcal{P}_n$ in \eqref{def.P_n}, one can study it in each Fourier mode.

As mentioned in Remark \ref{rem.thm.main} (\ref{item2.rem.thm.main}), the presence of the parameter $\gamma>2$ stabilizes the decay of solutions of \eqref{eq.S.tilde.intro}. A sharp contrast with \cite{HillairetWittwer13} is that this effect is exerted on the zero mode of the velocity $\mathcal{P}_0v=v_{r,0}(r){\bf e}_r+v_{\theta,0}(r){\bf e}_\theta$. Indeed, while $v_{r,0}(r)=0$ follows from the second and third lines in \eqref{eq.S.tilde.intro}, $v_{\theta,0}(r)$ satisfies the ordinary differential equation 
\begin{align}\label{eq.S.zeromode}
-\frac{\dd^2 v_{\theta,0}}{\dd r^2} 
-\frac{1+\gamma}r \frac{\dd v_{\theta,0}}{\dd r}  
+\frac{1-\gamma}{r^2} v_{\theta,0}
= f_{\theta,0}, \quad r>1, 
\end{align}
which is independent of $\alpha$. Then, since the Green functions of equation are $r^{-\gamma+1}$ and $r^{-1}$, the solution $v_{\theta,0}(r)$ decays subcritically if the data $f_{\theta,0}(r)$ decays fast enough.

The stabilization from $\gamma$ displays also in the non-zero modes. However, to prove Theorem \ref{thm.main}, we need to make precise the relationship between the decay of solutions and that of external forces. For this purpose, we derive the representation formula for the $n$-mode of the velocity $\mathcal{P}_nv=v_{r,n}(r)e^{i n \theta}{\bf e}_r+v_{\theta,n}(r)e^{i n \theta}{\bf e}_\theta$ using the streamfunction-vorticity variant of the system in \cite{HillairetWittwer13} and the Biot-Savart law in \cite{Maekawa17i, Maekawa17ii, GallagherHigakiMaekawa19, Higaki19}. When verifying the formula, we essentially use the decay of the vorticity subcritically improved by $\gamma$.

This paper is organized as follows. In Section \ref{sec.preliminaries}, we collect preliminary results from vector calculus in the polar coordinates. In Section \ref{sec.lin.}, we study the linear problem \eqref{eq.S.tilde.intro} decomposed into the Fourier modes. In Section \ref{sec.nonlin.}, we prove Theorem \ref{thm.main}.

{\bf Notations.} We denote by $C$ the constant and by $C(a,b,c,\ldots)$ the constant depending on $a,b,c,\ldots$. Both of these may vary from line to line. We use the function spaces 
$$\widehat{W}^{1,2}(\Omega)
=\{p\in L^2_{{\rm loc}}(\overline{\Omega})~|~\nabla p\in L^2(\Omega)^2\}, \quad
C^\infty_{0,\sigma}(\Omega)
=\{\varphi\in C^\infty_0(\Omega)^2~|~{\rm div}\,\varphi=0\}$$
and $L^2_{\sigma}(\Omega)$ which is the completion of $C^\infty_{0,\sigma}(\Omega)$ in the $L^2$-norm. If there is no confusion, we use the same notation to denote the quantities concerning scalar-, vector- or tensor-valued functions. For example, $\langle \cdot, \cdot\rangle$ denotes to the inner product on $L^2(\Omega)$, $L^2(\Omega)^2$ or $L^2(\Omega)^{2\times2}$.

\section{Preliminaries}\label{sec.preliminaries}
This section collects useful facts about the vector calculus in the exterior disk $\Omega$.
%
\subsection{Operators in the polar coordinates}\label{subsec.calc.in.polar.coord.}
%
The following formulas will be used: 
\begin{align}\label{formulas.polar}
\begin{split}
{\rm div}\,v 
& = \partial_1 v_1 + \partial_2 v_2 
=  \frac1r \partial_r (r v_r) + \frac1r \partial_\theta v_\theta, \\
{\rm rot}\,v 
& = \partial_1 v_2 - \partial_2 v_1 
=  \frac1r \partial_r (r v_\theta) - \frac1r \partial_\theta v_r,  \\
|\nabla v |^2 
& = 
|\partial_r v_r|^2 + |\partial_r v_\theta|^2 
+ \frac{1}{r^2} 
(|\partial_\theta v_r - v_\theta |^2 + |v_r + \partial_\theta v_\theta |^2), 
\end{split}
\end{align}
and
\begin{align*}
\begin{split}
-\Delta v 
& = \Big\{ -\partial_r \Big( \frac1r \partial_r (r v_r ) \Big)  
- \frac{1}{r^2} \partial_\theta^2 v_r 
+ \frac{2}{r^2} \partial_\theta v_\theta \Big\} {\bf e}_r \\
&\quad
+  \Big\{ - \partial_r \Big( \frac1r \partial_r (r v_\theta) \Big) 
- \frac{1}{r^2} \partial_\theta^2 v_\theta 
- \frac{2}{r^2} \partial_\theta v_r \Big\} {\bf e}_\theta.
\end{split}
\end{align*}
%

%
\subsection{Fourier series}\label{subsec.fourier.series}
%
Let $n\in \Z$. We define, for a vector field $v = v(r,\theta)$ on $\Omega$, 
\begin{align}\label{def.v_n}
v_n(r,\theta) = \mathcal{P}_n v(r,\theta), 
\end{align}
where $\mathcal{P}_n$ is the projection defined in \eqref{def.P_n}, for a scalar function $\omega=\omega(r,\theta)$ on $\Omega$, 
\begin{align}\label{def.omega_n}
\begin{split}
\mathcal{P}_n\omega(r,\theta)
&= 
\bigg( \frac{1}{2\pi} \int_0^{2\pi} \omega (r \cos s, r\sin s) e^{- i n s} \dd s \bigg) e^{i n \theta}, \\
\omega_n(r)
&= (\mathcal{P}_n \omega) e^{-i n \theta}, 
\end{split}
\end{align}
and for a function space $X(\Omega)\subset L^{1}_{\rm loc}(\overline{\Omega})^2$ or $X(\Omega)\subset L^{1}_{\rm loc}(\overline{\Omega})$, 
\begin{align*}
\mathcal{P}_n X(\Omega)
=\big\{\mathcal{P}_n f ~\big|~ f\in X(\Omega)\big\}. 
\end{align*}
Our definition of $f_n$ differs according to whether $f$ is vectorial or scalar. The former and latter are respectively defined in \eqref{def.v_n} as $f_n=\mathcal{P}_n f$ and in \eqref{def.omega_n} as $f_n=(\mathcal{P}_n f) e^{-i n \theta}$.

%
\subsection{Biot-Savart law}\label{subsec.biot-savart}
%
For a given $\omega\in L^\infty_2(\Omega)$, we consider the Poisson equation
\begin{equation*}
\left\{
\begin{array}{ll}
-\Delta \psi = \omega&\mbox{in}\ \Omega \\
\psi=0&\mbox{on}\ \partial\Omega. 
\end{array}\right.
\end{equation*}
The solution $\psi$ is called the streamfunction. Let $\omega\in\mathcal{P}_n L^\infty_2(\Omega)$ with $|n|\ge 1$ and set $\psi_n = (\mathcal{P}_n \psi) e^{-i n \theta}$ and $\omega_n = (\mathcal{P}_n \omega) e^{-i n \theta}$. In the polar coordinates, $\psi_n=\psi_n(r)$ solves the ordinary differential equation
\begin{align}\label{eq.streamfunc.}
-\frac{\dd^2 \psi_n}{\dd r^2}  
- \frac{1}{r} \frac{\dd \psi_n}{\dd r} 
+ \frac{n^2}{r^2} \psi_n 
= \omega_n,  \quad r>1,  
\qquad 
\psi_n (1) =0.
\end{align}
The decaying solution $\psi_n= \psi_n[\omega_n]$ of \eqref{eq.streamfunc.} is given by
\begin{align}\label{def.psi_n}
\begin{split}
\psi_n[\omega_n] (r) 
&= \frac{1}{2 |n|} 
\bigg(-d_n [\omega_n] r^{-|n|} \\
&\qquad\qquad
+ r^{-|n|} \int_1^r s^{|n|+1} \omega_n (s) \dd s   
+ r^{|n|}\int_r^\infty s^{-|n|+1} \omega_n (s) \dd s \bigg ), \\
d_n [\omega_n] 
&:= \int_1^\infty s^{-|n|+1} \omega_n (s) \dd s. 
\end{split}
\end{align}
Then the following vector field $V_n[\omega_n]$ is called the Biot-Savart law:
\begin{align}\label{def.Biot-Savart}
\begin{split}
&V_n [\omega_n](r,\theta)
= V_{r,n}[\omega_n](r) e^{i n \theta}{\bf e}_r  
+  V_{\theta,n}[\omega_n](r) e^{i n\theta} {\bf e}_\theta, \\
& V_{r,n} [\omega_n] 
= \frac{i n}{r} \psi_n [\omega_n], \qquad 
V_{\theta,n} [\omega_n]
=-\frac{\dd}{\dd r} \psi_n [\omega_n].
\end{split}
\end{align}
It is straightforward to see that
\begin{align}\label{properties.Biot-Savart}
\begin{split}
&{\rm div}\,V_n[\omega_n] = 0, 
\qquad
{\rm rot}\,V_n[\omega_n] = \omega_n e^{in\theta}, 
\qquad 
({\bf e}_r \cdot V_n [\omega_n])|_{\partial\Omega} = 0. 
\end{split}
\end{align}
If additionally $\omega\in L^\infty_\rho(\Omega)$ with $\rho>2$, one can check that $V_n [\omega_n]\in W^{1,2}(\Omega)^2$.

We state two propositions related to the Biot-Savart law. The first one is implicitly contained in \cite[Proposition 2.6 and Lemma 3.1]{Maekawa17i} and the second one has the same content as \cite[Corollary 2.7]{Maekawa17i}. However, we provide slightly more concise proofs for completeness.

%
\begin{proposition}\label{prop.biot-savart}
Let $|n|\ge1$ and $v_n\in \mathcal{P}_n W^{1,2}_0(\Omega)^2$. Set $\omega_n=({\rm rot}\,v_n)_n$. Then, if ${\rm div}\,v_n=0$ and $\omega_n\in L^\infty_\rho(\Omega)$ for some $\rho>2$, we have $v_n=V_n[\omega_n]$ and $d_n[\omega_n]=0$ in \eqref{def.psi_n}. 
\end{proposition}
%
%
\begin{proof}
Define $u=v_n-V_n[\omega_n]$. Then one has $u\in \mathcal{P}_n W^{1,2}(\Omega)^2$ and 
\begin{align*}
\begin{split}
&{\rm div}\,u = 0, 
\qquad
{\rm rot}\,u = 0, 
\qquad 
({\bf e}_r \cdot u)|_{\partial\Omega} = 0. 
\end{split}
\end{align*}
In particular, we have $\Delta u=0$ in the sense of distributions and hence $\|\nabla u\|_{L^2}=0$. Thus $u$ is a constant. Then $u_r=0$ follows from $({\bf e}_r \cdot u)|_{\partial\Omega} = 0$ and $u_\theta=0$ from ${\rm div}\,u=\frac1r \partial_r (r u_r) + \frac{in}r \partial_\theta u_\theta=0$. Consequently, we have $v_n= V_n[\omega_n]$ and, from \eqref{def.psi_n} and \eqref{def.Biot-Savart}, 
\begin{align*}
\begin{split}
d_n [\omega_n]
=\frac{\dd}{\dd r} \psi_n [\omega_n](1)
=-V_{\theta,n}[\omega_n](1)
=-v_{\theta,n}(1)=0.
\end{split}
\end{align*}
The proof is complete.
\end{proof}
%

%
\begin{proposition}\label{prop.rot.zero}
Let $|n|\ge1$ and $f_n\in \mathcal{P}_n L^{2}(\Omega)^2$. Then, if ${\rm rot}\,f_n = 0$ in the sense of distributions, we have $f=\nabla \mathcal{P}_n p$ for some $\mathcal{P}_n p\in \mathcal{P}_n \widehat{W}^{1,2}(\Omega)$. 
\end{proposition}
%
%
\begin{proof}
We only need to show that $\langle f_n, \varphi_n \rangle=0$ for all $\varphi_n\in \mathcal{P}_n C^\infty_{0,\sigma}(\Omega)$. Then the assertion follows from the Helmholtz decomposition in $L^2(\Omega)^2$. Note that, for $f_n\in \mathcal{P}_n C^\infty_0(\Omega)^2$ and $\mathcal{P}_n\psi\in \mathcal{P}_n C^\infty_0(\Omega)$, by integration by parts we have 
\begin{align*}
\begin{split}
\langle {\rm rot}\,f_n, \psi_n \rangle
&= 
\int_{0}^{2\pi} \int_{1}^{\infty}
\Big(
\frac{1}{r} \frac{\dd}{\dd r}(rf_{\theta,n}(r)) 
- \frac{in}{r} f_{r,n}(r)
\Big) \overline{\psi_n(r)} 
r \dd r \dd \theta \\
&= 
\int_{0}^{2\pi} \int_{1}^{\infty}
\Big(f_{r,n}(r) 
\overline{\frac{in}{r}\psi_n(r)}
- f_{\theta,n}(r) 
\overline{\frac{\dd}{\dd r} \psi_n(r)}
\Big) 
r \dd r \dd \theta. 
\end{split}
\end{align*}

Let $\varphi_n\in \mathcal{P}_n C^\infty_{0,\sigma}(\Omega)$ and set $\omega_n=({\rm rot}\,\varphi_n)_n$. By Proposition \ref{prop.biot-savart}, $\varphi_n$ is represented by the Biot-Savart law as $\varphi_n=V_n[\omega_n]$. Then the definition of $V_{r,n}[\omega_n](r)$ in \eqref{def.Biot-Savart} implies that the streamfunction $\psi(r,\theta)=\psi_n[\omega_n](r) e^{in\theta}$ is smooth and compactly supported in $\Omega$. Thus the condition that ${\rm rot}\,f_n = 0$ in the sense of distributions leads to 
\begin{align*}
\begin{split}
\int_{0}^{2\pi} \int_{1}^{\infty}
\Big(f_{r,n}(r) 
\overline{\frac{in}{r} \psi_n [\omega_n](r)}
- f_{\theta,n}(r) 
\overline{\frac{\dd}{\dd r} \psi_n [\omega_n](r)}
\Big) r \dd r \dd \theta
= 0. 
\end{split}
\end{align*}
The left-hand side can be rewritten as $\langle f_n, V_n[\omega_n] \rangle = \langle f_n, \varphi_n \rangle$. Consequently, we have $\langle f_n, \varphi_n \rangle=0$ for all $\varphi_n\in \mathcal{P}_n C^\infty_{0,\sigma}(\Omega)$. This completes the proof. 
\end{proof}
%

\section{Linearized problem}\label{sec.lin.}

In this section, we study the linearized problem 
\begin{equation}\tag{$\widetilde{\mbox{S}}$}\label{eq.S}
\left\{
\begin{array}{ll}
-\Delta v + (\alpha U - \gamma W)^\bot {\rm rot}\,v + \nabla q =f&\mbox{in}\ \Omega \\
{\rm div}\,v =0&\mbox{in}\ \Omega \\
v=0&\mbox{on}\ \partial\Omega \\
v(x)\to0&\mbox{as}\ |x|\to\infty.
\end{array}\right.
\end{equation}
Let $n\in\Z$.  Applying $\mathcal{P}_n$ in \eqref{def.P_n} to \eqref{eq.S}, we see that $(v_{r,n}(r),v_{\theta,n}(r))$ and $q_n(r)$ satisfy 
\begin{align}
&\begin{aligned}\label{eq.polar.vr}
&-\frac{\dd}{\dd r} \Big(\frac1r\frac{\dd}{\dd r} (r v_{r,n})\Big)
+ \frac{n^2}{r^2} v_{r,n}
+ \frac{2in}{r^2} v_{\theta,n} \\
&\qquad\qquad
- \frac{\alpha}{r^2} \Big(\frac{\dd}{\dd r} (r v_{\theta,n}) - in v_{r,n}\Big)
+ \partial_r q_n 
= f_{r,n}, \quad r>1, \\
\end{aligned} \\
&\begin{aligned}\label{eq.polar.vtheta}
&-\frac{\dd}{\dd r} \Big(\frac1r\frac{\dd}{\dd r} (r v_{\theta,n})\Big) 
+ \frac{n^2}{r^2} v_{\theta,n}
- \frac{2in}{r^2} v_{r,n} \\
&\qquad\qquad
- \frac{\gamma}{r^2} \Big(\frac{\dd}{\dd r} (r v_{\theta,n}) - in v_{r,n}\Big)
+ \frac{i n}{r} q_n 
= f_{\theta,n}, \quad r>1, 
\end{aligned}
\end{align}
the divergence-free and the no-slip boundary conditions
\begin{align}\label{eq.polar.div-free.no-slip}
\frac{\dd}{\dd r}(r v_{r,n}) + in v_{\theta,n}=0, 
\qquad v_{r,n}(1)=v_{\theta,n}(1)=0, 
\end{align}
and the condition at infinity
\begin{align}\label{eq.polar.spatial.infinity}
|v_{r,n}(r)| + |v_{\theta,n}(r)| \to 0, \quad r\to\infty. 
\end{align}
%

\subsection{Zero mode}\label{subsec.formulas.zero}

%
\begin{proposition}\label{prop.est.linear.zero}
For $\alpha\in\R$, $\gamma>2$, $2<\rho\le\gamma$ and $f=f_0\in \mathcal{P}_0 L^\infty_{2\rho-1}(\Omega)^2$, there is a unique solution $(v_0, \nabla\mathcal{P}_0 q)$ of \eqref{eq.S} with $v_0\in \mathcal{P}_0 L^2_\sigma(\Omega) \cap W^{1,2}_0(\Omega)^2 \cap W^{2,2}(\Omega)^2$ and $\mathcal{P}_0 q\in \mathcal{P}_0 \widehat{W}^{1,2}(\Omega)$ satisfying $v_0(r,\theta)=v_{\theta,0}(r) {\bf e}_\theta$ and 
\begin{align}
r^{\rho} |\omega_0(r)|
&\le
\frac{C}{\rho-2}
\|f_0\|_{L^\infty_{2\rho-1}}, 
\label{est1.prop.est.linear.zero} \\
r^{\rho-1} |v_{0}(r,\theta)|
+ \frac{1}{\gamma-1} r^{\rho} |\nabla v_{0}(r,\theta)|
&\le 
\frac{C}{(\gamma-2)(\rho-2)}
\|f_0\|_{L^\infty_{2\rho-1}}. 
\label{est2.prop.est.linear.zero}
\end{align}
Here $\omega_0:=({\rm rot}\,v_0)_0$. The constant $C$ is independent of $\alpha$, $\gamma$ and $\rho$.
\end{proposition}
%
%
\begin{proof}
Let $n=0$ in \eqref{eq.polar.vr}--\eqref{eq.polar.spatial.infinity}. The first condition in \eqref{eq.polar.div-free.no-slip} leads to that $v_{r,0}(r)=\frac{C}{r}$ with some constant $C$. 
Then the second condition leads to $C=0$, which yields that $v_{r,0}=0$.

Thus we focus on the angular part $v_{\theta,0}=v_{\theta,0}(r)$. From \eqref{eq.polar.vtheta} and \eqref{eq.polar.div-free.no-slip}, we find that 
\begin{align}\label{ode.vel.0Fourier}
-\frac{\dd^2 v_{\theta,0}}{\dd r^2} 
-\frac{1+\gamma}r \frac{\dd v_{\theta,0}}{\dd r}  
+\frac{1-\gamma}{r^2} v_{\theta,0}
= f_{\theta,0}, \quad r>1,
\qquad v_{\theta,0}(1)=0.
\end{align}
The linearly independent solutions of the homogeneous equation of \eqref{ode.vel.0Fourier} are
\begin{equation*}
\begin{aligned}
r^{-\gamma+1} \quad \text{and} \quad r^{-1}, 
\end{aligned}
\end{equation*}
and their Wronskian is $(\gamma-2)r^{-\gamma-1}$. Hence the solution of \eqref{ode.vel.0Fourier} in $L^2(\Omega)$ is given by
\begin{equation*}
\begin{split}
v_{\theta,0}(r)
&= 
\frac{1}{\gamma-2}
\bigg\{	
-\Big(\int_{1}^{\infty} s^{2} f_{\theta,0}(s)\dd s \Big) r^{-\gamma+1} \\
&\qquad\qquad\qquad
+ r^{-\gamma+1} \int_{1}^{r} s^{\gamma} f_{\theta,0}(s)\dd s 
+ r^{-1} \int_{r}^{\infty} s^{2} f_{\theta,0}(s)\dd s
\bigg\}. 
\end{split}
\end{equation*}
Then we have
\begin{align*}
\begin{split}
\omega_0(r)
= 
\frac1r\frac{\dd}{\dd r} (r v_{\theta,0})(r)
=
\Big(\int_{1}^{\infty} s^{2} f_{\theta,0}(s)\dd s \Big) r^{-\gamma} - r^{-\gamma} \int_{1}^{r} s^{\gamma} f_{\theta,0}(s)\dd s. 
\end{split}
\end{align*}
By the assumption $f\in \mathcal{P}_0 L^\infty_{2\rho-1}(\Omega)^2$ with $2<\rho\le\gamma$ and the computations
\begin{align*}
r^{-\gamma+1} \int_{1}^{r} s^{\gamma-2\rho+1} \dd s
\le \frac{1}{\rho-2} r^{-\rho+1}, 
\qquad
r^{-1} \int_{r}^{\infty} s^{-2\rho+3} \dd s
\le \frac{1}{\rho-2} r^{-\rho+1}, 
\end{align*}
one can check that $\omega_0$ satisfies \eqref{est1.prop.est.linear.zero} and that $v_0=v_{\theta,0} {\bf e}_\theta$ belongs to $\mathcal{P}_0 L^2_\sigma(\Omega) \cap W^{1,2}_0(\Omega)^2 \cap W^{2,2}(\Omega)^2$ and satisfies \eqref{est2.prop.est.linear.zero}. The pressure $\mathcal{P}_0 q \in \widehat{W}^{1,2}(\Omega)$ is obtained by \eqref{eq.polar.vr}. Clearly, $(v_0, \nabla \mathcal{P}_0 q)$ is the unique solution of \eqref{eq.polar.vr}--\eqref{eq.polar.spatial.infinity}. The proof is complete. 
\end{proof}
%

\subsection{Non-zero modes}\label{subsec.formulas.nonzero}

By Proposition \ref{prop.est.linear.zero}, the zero mode of the solution of \eqref{eq.S} decays as fast as desired, if the external force does correspondingly. On the other hand, for the non-zero mode, the decay rate is governed by the Biot-Savart law \eqref{def.Biot-Savart}. Taking this into account, we will build a solution $v_n$ of the non-zero mode of \eqref{eq.S} satisfying, for $2<\rho\le\min\{\gamma,3\}$, 
\begin{align}\label{cond.nonlin.1}
{\rm rot}\,v_n = (|x|^{-\rho}), \qquad v_n = O(|x|^{-\rho+1}), 
\end{align}
under suitable assumptions on the external force $f_n$.

For $|n|\ge1$, we define 
\begin{align}\label{def.zeta.etc.}
n_\gamma = \Big\{n^2 + \Big(\frac{\gamma}{2}\Big)^2\Big\}^{\frac12}, \qquad 
\zeta_n = (n_\gamma^2 + i\alpha n)^{\frac12}. 
\end{align}
This definition of $\zeta_n$ coincides with the one in \cite{HillairetWittwer13} if $\gamma=0$. We compute 
\begin{align}\label{modulus.re.im.zeta}
\begin{split}
|\zeta_n| 
&= 
n_\gamma \Big\{1 + \Big(\frac{\alpha n}{n_\gamma^2}\Big)^2 \Big\}^{\frac14}, \\
\Re(\zeta_n) 
&= 
\frac{n_\gamma}{\sqrt{2}} 
\bigg[\Big\{1 + \Big(\frac{\alpha n}{n_\gamma^2}\Big)^2\Big\}^{\frac12} + 1 \bigg]^{\frac12}, \\
\Im(\zeta_n) 
&= 
{\rm sgn}(\alpha n)
\frac{n_\gamma}{\sqrt{2}} 
\bigg[\Big\{1 + \Big(\frac{\alpha n}{n_\gamma^2}\Big)^2\Big\}^{\frac12} - 1 \bigg]^{\frac12}. 
\end{split}
\end{align}
Let us set $\xi_n=\Re(\zeta_n)$ for simplicity. Then we have 
\begin{align}\label{ineqs.xi}
\xi_n \le |\zeta_n| \le \sqrt{2} \xi_n, 
\qquad 
\frac{\xi_n}{|n|} 
\le 
(|\alpha|^{\frac12} + \gamma), 
\qquad 
0<\Big(\xi_n - \frac{\gamma}2\Big)^{-1}
< 2\gamma
\end{align}
with $C$ independent of $n$, $\alpha$ and $\gamma$.

%
\begin{proposition}\label{prop.est1.linear.nonzero}
For $|n|\ge1$, $\alpha\in\R$, $\gamma>2$, $2<\rho\le3$ with $2<\rho\le\min\{\gamma,3\}$ and $f=f_n\in \mathcal{P}_n L^\infty_{2\rho-1}(\Omega)^2$, there is a unique solution $(v_n,\nabla\mathcal{P}_n q)$ of \eqref{eq.S} with $v_n\in \mathcal{P}_n L^2_\sigma(\Omega) \cap W^{1,2}_0(\Omega)^2 \cap W^{2,2}_{{\rm loc}}(\overline{\Omega})^2$ and $\mathcal{P}_n q \in \mathcal{P}_n W^{1,2}_{{\rm loc}}(\overline{\Omega})$ satisfying the following. 
\begin{enumerate}[(1)]
\item For $|n|\ge1$,  
\begin{align}\label{est1.prop.est1.linear.nonzero}
r^{\rho} |\omega_n(r)|
\le 
\frac{C}{|n|}
\Big(\xi_n+\frac{\gamma}2\Big) 
\Big(\xi_n-\frac{\gamma}2\Big)^{-1} 
\|f_n\|_{L^\infty_{2\rho-1}}.
\end{align}

\item If $2<\rho\le\gamma<3$ or $2<\rho<3\le\gamma$, for $|n|\ge1$,
\begin{align}\label{est2.prop.est1.linear.nonzero}
\begin{split}
& r^{\rho-1} |v_n(r,\theta)| 
+ \frac{r^{\rho}}{|n|} |\nabla v_n(r,\theta)| \\
&\le 
\frac{C}{|n|(|n|-\rho+2)} 
\Big(\xi_n+\frac{\gamma}2\Big) 
\Big(\xi_n-\frac{\gamma}2\Big)^{-1} 
\|f_n\|_{L^\infty_{2\rho-1}}.
\end{split}
\end{align}

\item If $\gamma\ge3$ and $\rho=3$, for $|n|=1$,
\begin{align}\label{est3.prop.est1.linear.nonzero}
\begin{split}
&r^2 (\log r)^{-1} |v_n(r,\theta)|
+ r^3 (\log r)^{-1} |\nabla v_n(r,\theta)| \\
&\le 
C 
\Big(\xi_n+\frac{\gamma}2\Big) 
\Big(\xi_n-\frac{\gamma}2\Big)^{-1} 
\|f_n\|_{L^\infty_{5}}, 
\end{split}
\end{align}
and for $|n|>1$, 
\begin{align}\label{est4.prop.est1.linear.nonzero}
\begin{split}
r^{2} |v_n(r,\theta)|
+ \frac{r^{3}}{|n|} |\nabla v_n(r,\theta)| 
\le 
\frac{C}{|n|^2} 
\Big(\xi_n+\frac{\gamma}2\Big) 
\Big(\xi_n-\frac{\gamma}2\Big)^{-1} 
\|f_n\|_{L^\infty_{5}}.
\end{split}
\end{align}
\end{enumerate}
Here $\omega_n:=({\rm rot}\,v_n)_n$. The constant $C$ is independent of $n$, $\alpha$, $\gamma$ and $\rho$.
\end{proposition}
%
%
\begin{proof}
The proof is divided into two parts. First we show the existence of solutions using the representation formula. Second we verify the uniqueness by Proposition \ref{prop.biot-savart}.

({\bf Existence}) Initially, let us assume that $f_n\in \mathcal{P}_n C^\infty_0(\Omega)^2$. Operating ${\rm rot}\,$ and $\mathcal{P}_n$ to the first line of \eqref{eq.S}, we see that $\omega_n=\omega_n(r)$ satisfies the ordinary differential equation 
\begin{align}\label{ode1.vor.nFourier}
-\frac{\dd^2 \omega_n}{\dd r^2} 
- \frac{1+\gamma}{r} \frac{\dd \omega_n}{\dd r} 
+ \frac{n^2+i\alpha n}{r^2} \omega_n 
= ({\rm rot}\,f_n)_n, \quad r>1. 
\end{align}
By the transformation
\begin{align*}
\omega_n(r) = r^{-\frac{\gamma}2} \tilde{\omega}_n(r), 
\end{align*}
we find that $\tilde{\omega}_n$ solves 
\begin{align}\label{ode2.vor.nFourier}
-\frac{\dd^2 \tilde{\omega}_n}{\dd r^2} 
- \frac{1}{r} \frac{\dd \tilde{\omega}_n}{\dd r} 
+ \frac{\zeta_n^2}{r^2} \tilde{\omega}_n
= r^{\frac{\gamma}2} ({\rm rot}\,f_n)_n, \quad r>1. 
\end{align}
The linearly independent solutions of the homogeneous equation of \eqref{ode2.vor.nFourier} are
\begin{align*}
r^{-\zeta_n} \quad \text{and} \quad  r^{\zeta_n},
\end{align*}
and their Wronskian is $2\zeta_n r^{-1}$. Hence the decaying solution of \eqref{ode2.vor.nFourier} is given by
\begin{align}\label{rep1.vol.nFourier}
\tilde{\omega}_n(r)
&= 
c_n[f_n] r^{-\zeta_n} + \Psi_{n}[f_n](r), 
\end{align}
where the constant $c_n[f_n]$ is to be determined later and $\Psi_{n}[f_n]$ is defined by 
\begin{align}\label{rep2.vol.nFourier}
\begin{split}
\Psi_{n}[f_n](r)
& = 
\frac{r^{-\zeta_n}}{\zeta_n} 
\int_1^r s^{\zeta_n + \frac{\gamma}{2} + 1} ({\rm rot}\,f_n)_n(s) \dd s \\
&\quad
+ \frac{r^{\zeta_n}}{\zeta_n}
\int_r^\infty s^{-\zeta_n + \frac{\gamma}{2} + 1} ({\rm rot}\,f_n)_n(s) \dd s \\
& =
-\frac{r^{-\zeta_n}}{\zeta_n} 
\int_1^r s^{\zeta_n + \frac{\gamma}2} 
\Big\{\Big(\zeta_n + \frac{\gamma}2\Big) f_{\theta,n}(s) + inf_{r,n}(s)\Big\} \dd s \\
&\quad
+ \frac{r^{\zeta_n}}{\zeta_n}
\int_r^\infty s^{-\zeta_n + \frac{\gamma}2} 
\Big\{\Big(\zeta_n - \frac{\gamma}2\Big) f_{\theta,n}(s) - inf_{r,n}(s)\Big\} \dd s. 
\end{split}
\end{align}
Here we performed integration by parts using $({\rm rot}\,f_n)_n(r) = \frac1r \frac{\dd}{\dd r} (r f_{\theta,n})(r) - \frac{i n}{r} f_{r,n}(r)$. Going back to the equation \eqref{ode1.vor.nFourier}, we see that the decaying solution is given by 
\begin{align}\label{rep3.vol.nFourier}
\omega_n(r) = c_n[f_n] r^{-\zeta_n-\frac{\gamma}2} + \Phi_{n}[f_n](r), \qquad 
\Phi_{n}[f_n](r):= r^{-\frac{\gamma}2} \Psi_{n}[f_n](r). 
\end{align}

Let us determine $c_n[f_n]$ in \eqref{rep1.vol.nFourier}. From $f_n\in \mathcal{P}_n C^\infty_0(\Omega)^2$, we have $|\Phi_{n}[f_n](r)|\le C(f,n) r^{-\xi_n-\frac{\gamma}2}$. We choose $c_n[f_n]$ so that $d_n [\omega_n]$ in \eqref{def.psi_n} is zero, namely, 
\begin{align}\label{rep4.vol.nFourier}
\begin{split}
c_n[f_n]
= -\Big(\zeta_n + |n| + \frac{\gamma}2 - 2\Big) 
\int_1^\infty s^{-|n|+1} \Phi_{n}[f_n](s) \dd s.  
\end{split}
\end{align}
Then the Biot-Savart law $V_n [\omega_n]=:v_n$ in \eqref{def.Biot-Savart} is written as 
\begin{align}\label{rep1.vel.nFourier}
\begin{split}
v_n(r,\theta)
&=v_{r,n}(r) e^{in\theta} {\bf e}_r + v_{r,n}(r) e^{in\theta} {\bf e}_\theta, \\
v_{r,n}(r)
&=\frac{in}{2|n|}
\Big(r^{-|n|-1} \int_{1}^{r} s^{|n|+1} \omega_n(s)\dd s + r^{|n|-1} \int_{r}^{\infty} s^{-|n|+1} \omega_n(s)\dd s
\Big), \\
v_{\theta,n}(r)
&=\frac{1}{2}
\Big(r^{-|n|-1} \int_{1}^{r} s^{|n|+1} \omega_n(s)\dd s - r^{|n|-1} \int_{r}^{\infty} s^{-|n|+1} \omega_n(s)\dd s
\Big). 
\end{split}
\end{align}

Let us show that, for $v_n$ in \eqref{rep1.vel.nFourier}, there is a pressure $\mathcal{P}_n q\in \mathcal{P}_n\widehat{W}^{1,2}(\Omega)$ such that the pair $(v_n,\nabla\mathcal{P}_n q)$ is a solution of \eqref{eq.S}. From $f_n\in \mathcal{P}_n C^\infty_0(\Omega)^2$ and $\xi_n > \frac{\gamma}2$ implied by \eqref{ineqs.xi}, we see that $\omega_n$ in \eqref{rep3.vol.nFourier} is smooth and satisfies $|\nabla^k \omega_n(r)|\le C(f,n,k) r^{-\min\{\gamma,3\}}$ for any $k\in\Z_{\ge0}$. This, combined with the choice of $c_n[f_n]$ in \eqref{rep4.vol.nFourier} and Lemma \ref{lem.app.est.int.1}, ensures that $v_n$ is smooth and satisfies $v_n|_{\partial\Omega}=0$ and $\|\nabla^k v_n\|_{L^2}<\infty$ for any $k\in\Z_{\ge0}$. Accordingly, we can apply Proposition \ref{prop.rot.zero} because of $({\rm rot}\,v_n)_n=\omega_n$ and 
\begin{align*}
\begin{split}
&\big(
{\rm rot}\,(-\Delta v_n + (\alpha U - \gamma W)^\bot {\rm rot}\,v_n - f_n)
\big)_n \\
&=-\frac{\dd^2 \omega_n}{\dd r^2} 
- \frac{1+\gamma}{r} \frac{\dd \omega_n}{\dd r} 
+ \frac{n^2+i\alpha n}{r^2} \omega_n 
- ({\rm rot}\,f_n)_n
=0. 
\end{split}
\end{align*}
Hence there is a pressure $\mathcal{P}_n q\in\mathcal{P}_n\widehat{W}^{1,2}(\Omega)$ such that $(v_n,\nabla\mathcal{P}_n q)$ is a solution of \eqref{eq.S}.

Now, let $f_n\in \mathcal{P}_n L^\infty_{2\rho-1}(\Omega)^2$ in \eqref{rep2.vol.nFourier}--\eqref{rep1.vel.nFourier}. We will prove the estimates \eqref{est1.prop.est1.linear.nonzero}--\eqref{est4.prop.est1.linear.nonzero}. For $\Phi_{n}[f_n]$ in \eqref{rep3.vol.nFourier} and $c_n[f_n]$ in \eqref{rep4.vol.nFourier}, using Lemma \ref{lem.app.est.int.1}, we have 
\begin{align*}
\begin{split}
\|\Phi_{n}[f_n]\|_{L^\infty_\rho}
+ |c_n[f_n]|
&\le 
C\bigg(
1 + \frac{|\zeta_n|+|n|+\dfrac{\gamma}2-2}{|n|+\rho-2} \bigg)
\|\Phi_{n}[f_n]\|_{L^\infty_\rho} \\
&\le 
\frac{C}{|n|}
\Big(\xi_n+\frac{\gamma}2\Big) 
\Big(\xi_n-\frac{\gamma}2\Big)^{-1} 
\|f_n\|_{L^\infty_{2\rho-1}}, 
\end{split}
\end{align*}
where \eqref{ineqs.xi} is used in the second inequality and $C$ is independent of $n$, $\alpha$, $\gamma$ and $\rho$. Hence we see from \eqref{rep3.vol.nFourier} that $\omega_n$ satisfies \eqref{est1.prop.est1.linear.nonzero}. For $v_n$ in \eqref{rep1.vel.nFourier}, using \eqref{formulas.polar} and  \eqref{est1.prop.est1.linear.nonzero}, we have
\begin{align*}
\begin{split}
&|v_n(r,\theta)| 
+ \frac{r}{|n|} |\nabla v_n(r,\theta)| \\
&\le 
\frac{C}{|n|} 
\Big(\xi_n+\frac{\gamma}2\Big) 
\Big(\xi_n-\frac{\gamma}2\Big)^{-1} 
\|f_n\|_{L^\infty_{2\rho-1}} \\
&\qquad
\times\bigg(r^{-|n|-1} \int_{1}^{r} s^{|n|-\rho+1} \dd s 
+ r^{|n|-1} \int_{r}^{\infty} s^{-|n|-\rho+1} \dd s
\bigg). 
\end{split}
\end{align*}
Hence we see from Lemma \ref{lem.app.est.int.1} (\ref{item2.lem.app.est.int.1}) that $v_n$ satisfies \eqref{est2.prop.est1.linear.nonzero}--\eqref{est4.prop.est1.linear.nonzero}.

It remains to verify that $v_n$ given by \eqref{rep1.vel.nFourier} is a solution of \eqref{eq.S} for general $f_n\in \mathcal{P}_n L^\infty_{2\rho-1}(\Omega)^2$. The condition $v_n\in W^{1,2}_0(\Omega)^2$ follows from \eqref{est2.prop.est1.linear.nonzero}--\eqref{est4.prop.est1.linear.nonzero} and the choice of $c_n[f_n]$ in \eqref{rep4.vol.nFourier}. By Lemma \ref{lem.app.est.int.2} and direct computation, one has 
\begin{align}\label{est1.proof.prop.est1.linear.nonzero}
\begin{split}
& \|\Phi_{n}[f_n]\|_{L^\infty_2} 
+ |c_n[f_n]| 
+ \|\omega_n\|_{L^\infty_2} 
+ \|\nabla v_n\|_{L^2} \\
&\le 
C(n,\alpha,\gamma) 
\|\mu f_n\|_{L^2}, 
\qquad \mu(x):=|x|^2. 
\end{split}
\end{align}
To show that $v_n$ satisfies the weak formulation of \eqref{eq.S}, we take $\{\psi_{n}^{(m)}\}_{m=1}^\infty\subset \mathcal{P}_nC^\infty_0(\Omega)^2$ such that 
$\displaystyle{\lim_{m\to\infty}\|\mu(f_n-\psi_{n}^{(m)})\|_{L^2}}=0$. Let $v_{n}^{(m)}$ denote the smooth solution given by \eqref{rep1.vel.nFourier} replacing $f_n$ by $\psi_{n}^{(m)}$. Also, let $\mathcal{P}_n q_{n}^{(m)}\in \mathcal{P}_n \widehat{W}^{1,2}(\Omega)$ be an associated pressure. Then, using \eqref{est1.proof.prop.est1.linear.nonzero} and linearity of the equations, we have, for any $\varphi\in C^\infty_{0,\sigma}(\Omega)$, 
\begin{align*}
\begin{split}
&\langle\nabla v_n, \nabla \varphi \rangle
+ \langle (\alpha U - \gamma W)^\bot{\rm rot}\,v_n - f_n, \varphi \rangle \\
&= \lim_{m\to\infty} 
\langle -\Delta v_{n}^{(m)} + (\alpha U - \gamma W)^\bot{\rm rot}\,v_{n}^{(m)} - \psi_{n}^{(m)}, \varphi \rangle \\
&= \lim_{m\to\infty} 
\langle \nabla \mathcal{P}_n q_{n}^{(m)}, \varphi \rangle = 0. 
\end{split}
\end{align*}
Here we performed integration by parts in the first and last equalities. Consequently, we see that $v_n$ is a weak solution of \eqref{eq.S}. Then, by regarding $(\alpha U - \gamma W)^\bot{\rm rot}\,v_n - f_n$ as the external force, one can prove the local regularity $v_n\in W^{2,2}_{{\rm loc}}(\overline{\Omega})^2$ and the existence of an associated pressure $\mathcal{P}_n q \in \mathcal{P}_n W^{1,2}_{{\rm loc}}(\overline{\Omega})$ by standard theory for the Stokes system; see Sohr \cite[Chapter I\hspace{-.1em}I\hspace{-.1em}I]{Sohr2013} for example. This completes the existence part of the proof.

({\bf Uniqueness}) Assume that $(v_n,\nabla\mathcal{P}_n q)\in \big(\mathcal{P}_n L^2_\sigma(\Omega) \cap W^{1,2}_{0}(\Omega)^2 \cap W^{2,2}_{{\rm loc}}(\overline{\Omega})^2\big)\times \mathcal{P}_n L^{2}_{{\rm loc}}(\overline{\Omega})^2$ is a solution of \eqref{eq.S} with $f=0$. By the elliptic regularity, $v_n$ is smooth in $\Omega$. Since $\omega_n=({\rm rot}\,v_n)_n$ satisfies the homogeneous equation of \eqref{ode1.vor.nFourier}, due to the summability $\nabla v_n\in L^2(\Omega)^{2\times2}$, we have $\omega_n=\tilde{c}_n r^{-\zeta_n-\frac{\gamma}2}$ with some constant $\tilde{c}_n$. Then Proposition \ref{prop.biot-savart} implies that $v_n=V_n[\omega_n]$ and $d_n [\omega_n]=0$. The latter condition is rewritten as 
\begin{align*}
0
=d_n[\omega_n]
=\tilde{c}_n \int_1^\infty s^{-\zeta_n-|n|-\frac{\gamma}2+1} \dd s
=\frac{\tilde{c}_n}{\zeta_n+|n|+\dfrac{\gamma}2-2}.
\end{align*}
Thus we have $\tilde{c}_n=0$ and hence $\omega_n=0$. Then the uniqueness follows from $v_n=V_n[\omega_n]=0$. This completes the uniqueness part of the proof. Hence we conclude. 
\end{proof}
%

\section{Nonlinear problem}\label{sec.nonlin.}

In this section we prove Theorem \ref{thm.main} by showing the unique solvability of the system 
\begin{equation}\tag{$\widetilde{\mbox{NS}}$}\label{eq.NS.tilde}
\left\{
\begin{array}{ll}
-\Delta v + (\alpha U - \gamma W)^\bot {\rm rot}\,v + \nabla q 
= - v^{\bot} {\rm rot}\,v + f &\mbox{in}\ \Omega \\
{\rm div}\,v =0&\mbox{in}\ \Omega \\
v=0&\mbox{on}\ \partial\Omega \\
v(x)\to0&\mbox{as}\ |x|\to\infty. 
\end{array}\right.
\end{equation}
Let us collect two lemmas needed in the proof. For $\rho\ge0$, we define the Banach space
\begin{align*}
l^1\big(L^\infty_{\rho}(\Omega)\big)
&= \bigg\{f=\sum_{n\in\Z}\mathcal{P}_n f~\bigg|~ 
\|f\|_{l^1L^\infty_\rho}:=\sum_{n\in\Z} \|\mathcal{P}_n f\|_{L^\infty_\rho}<\infty\bigg\}.
\end{align*}
%
%
\begin{lemma}\label{lem.est.l1}
For $\alpha\in\R$, $\gamma>2$, $2<\rho<3$ with $2<\rho\le\min\{\gamma,3\}$ and $f\in l^1\big(L^\infty_{2\rho-1}(\Omega)\big)^2$, there is a unique solution $(v,\nabla q)$ of \eqref{eq.S} with $v\in L^2_\sigma(\Omega) \cap W^{1,2}_0(\Omega)^2 \cap W^{2,2}_{{\rm loc}}(\overline{\Omega})^2$ and $q \in W^{1,2}_{{\rm loc}}(\overline{\Omega})$ satisfying
\begin{align}\label{est1.cor.est.l1}
\begin{split}
\|v\|_{l^1 L^\infty_{\rho-1}} + \|\nabla v\|_{l^1 L^\infty_{\rho}} 
\le 
\kappa \|f\|_{l^1 L^\infty_{2\rho-1}}, 
\qquad 
\kappa=\kappa(\alpha,\gamma,\rho):=\frac{C_0(|\alpha|^\frac12+\gamma) \gamma}{(\rho-2)^2(3-\rho)}. 
\end{split}
\end{align}
The constant $C_0$ is independent of $\alpha$, $\gamma$ and $\rho$.
\end{lemma}
%
%
\begin{proof}
This follows from Propositions \ref{prop.est.linear.zero} and \ref{prop.est1.linear.nonzero} and an estimate obtained from \eqref{ineqs.xi}: 
\begin{align*}
\begin{split}
\frac{1}{|n|}
\Big(\xi_n+\frac{\gamma}2\Big) 
\Big(\xi_n-\frac{\gamma}2\Big)^{-1} 
\le 4(|\alpha|^\frac12+\gamma) \gamma. 
\end{split}
\end{align*}
Indeed, we see from the propositions that there is a weak solution $v\in L^2_\sigma(\Omega) \cap W^{1,2}_0(\Omega)^2$ of \eqref{eq.S} satisfying \eqref{est1.cor.est.l1}. Then, as in the proof of Proposition \ref{prop.est1.linear.nonzero}, we have the local regularity $v\in W^{2,2}_{{\rm loc}}(\overline{\Omega})^2$ and an associate pressure $q \in W^{1,2}_{{\rm loc}}(\overline{\Omega})$. The uniqueness of $(v,\nabla q)$ follows from that of $(v_n,\nabla \mathcal{P}_n q)$ for $n\in\Z$. This completes the proof. 
\end{proof}
%

%
\begin{lemma}\label{lem.est.nonlinear}
For $v\in l^1\big(L^\infty_{\gamma_1}(\Omega)\big)^2$ and $\omega\in l^1\big(L^\infty_{\gamma_2}(\Omega)\big)$, we have
\begin{align*}
\begin{split}
\|v \omega\|_{l^1L^\infty_{\gamma_1+\gamma_2}}
\le
\|v\|_{l^1L^\infty_{\gamma_1}} 
\|\omega\|_{l^1L^\infty_{\gamma_2}}. 
\end{split}
\end{align*}
\end{lemma}
%
%
\begin{proof}
The identity 
\begin{align*}
\mathcal{P}_n (v \omega)
&=
\sum_{m\in\Z} 
(\mathcal{P}_m v) (\mathcal{P}_{n-m} \omega), \quad n\in\Z 
\end{align*}
and the Young inequality for sequences imply 
\begin{align*}
\|v \omega\|_{l^1L^\infty_{\gamma_1+\gamma_2}} 
&\le
\sum_{n\in\Z} \sum_{m\in\Z}
\|\mathcal{P}_m v\|_{L^\infty_{\gamma_1}} 
\|\mathcal{P}_{n-m} \omega\|_{L^\infty_{\gamma_2}} \\
&\le
\|v\|_{l^1L^\infty_{\gamma_1}}
\|\omega\|_{l^1L^\infty_{\gamma_2}}. 
\end{align*}
This completes the proof. 
\end{proof}
%

%
\begin{proofx}{Theorem \ref{thm.main}}
Firstly we show the unique solvability of \eqref{eq.NS.tilde} under smallness conditions on $f$. We define the Banach space 
\begin{align*}
\mathcal{X}_{\rho}
=\Big\{w\in L^2_\sigma(\Omega) \cap W^{1,2}_0(\Omega)^2 \cap l^1\big(L^\infty_{\rho-1}(\Omega)\big)^2 
~\Big|~ 
\nabla w \in l^1\big(L^\infty_{\rho}(\Omega)\big)^{2\times2} \Big\},  
\end{align*}
equipped with the norm $\|w\|_{\mathcal{X}_{\rho}}:=\|w\|_{l^1 L^\infty_{\rho-1}}+\|\nabla w\|_{l^1 L^\infty_{\rho}}$, and consider the closed subset 
\begin{align*}
\mathcal{B}_{\rho}(\delta) 
= \{w\in \mathcal{X}_{\rho} ~|~ \|w\|_{\mathcal{X}_{\rho}} \le \delta \}, 
\quad \delta>0. 
\end{align*}

For any $w\in \mathcal{X}_{\rho}$, by Lemma \ref{lem.est.l1}, there is a unique solution $(v_w,\nabla q_w)$ to the problem 
\begin{equation*}
\left\{
\begin{array}{ll}
-\Delta v_w + (\alpha U - \gamma W)^\bot {\rm rot}\,v_w + \nabla q_w = - w^{\bot} {\rm rot}\,w + f &\mbox{in}\ \Omega \\
{\rm div}\,v_w =0&\mbox{in}\ \Omega \\
v_w=0&\mbox{on}\ \partial\Omega \\
v_w(x)\to0&\mbox{as}\ |x|\to\infty 
\end{array}\right.
\end{equation*}
satisfying 
\begin{align}\label{est1.proof.thm.main}
\begin{split}
\|v_w\|_{\mathcal{X}_{\rho}}
=\|v_w\|_{l^1 L^\infty_{\rho-1}} 
+ \|\nabla v_w\|_{l^1 L^\infty_{\rho}} 
&\le 
\kappa
\big(\|w^{\bot} {\rm rot}\,w\|_{l^1 L^\infty_{2\rho-1}} + \|f\|_{l^1 L^\infty_{2\rho-1}}\big) \\
&\le 
\kappa
\big(\|w\|_{l^1 L^\infty_{\rho-1}} \|{\rm rot}\,w\|_{l^1 L^\infty_{\rho}} + \|f\|_{l^1 L^\infty_{2\rho-1}}\big), 
\end{split}
\end{align}
where Lemma \ref{lem.est.nonlinear} is applied in the second inequality. Hence, by denoting $v_w$ by $T(w)$, we see that $T$ defines a linear map from $\mathcal{X}_{\rho}$ to itself.

Let us show that $T$ is a contraction on $\mathcal{B}_{\rho}(\delta)\subset \mathcal{X}_{\rho}$ if both $\|f\|_{l^1 L^\infty_{2\rho-1}}$ and $\delta$ are sufficiently small depending on $\kappa=\kappa(\alpha,\gamma,\rho)$. Then the unique existence of solutions of \eqref{eq.NS.tilde} follows from the Banach fixed-point theorem. For any $w\in \mathcal{B}_{\rho}(\delta)$, by \eqref{est1.proof.thm.main}, we have 
\begin{align}\label{est2.proof.thm.main}
\begin{split}
\|T(w)\|_{\mathcal{X}_{\rho}}
&\le 
\kappa\big(\delta^2 + \|f\|_{l^1 L^\infty_{2\rho-1}}\big). 
\end{split}
\end{align}
On the other hand, for any $w_1,w_2\in\mathcal{B}_{\rho}(\delta)$, by Lemma \ref{lem.est.nonlinear} again, we have
\begin{align}\label{est3.proof.thm.main}
\begin{split}
\|T(w_2)-T(w_1)\|_{\mathcal{X}_{\rho}} 
&\le 
\kappa \|w_2^{\bot} {\rm rot}\,w_2 - w_1^{\bot} {\rm rot}\,w_1\|_{l^1 L^\infty_{2\rho-1}} \\
&\le 
2\kappa \delta \|w_2 - w_1\|_{\mathcal{X}_{\rho}}. 
\end{split}
\end{align}
Therefore, fixing $\delta$ and $\|f\|_{l^1 L^\infty_{2\rho-1}}$ so that 
\begin{align}\label{est4.proof.thm.main}
\begin{split}
\delta < \frac{1}{2\widetilde{\kappa}}
\quad \text{and} \quad 
\|f\|_{l^1 L^\infty_{2\rho-1}} \le \delta^2, 
\end{split}
\end{align}
we see that $T:\mathcal{B}_{\rho}(\delta)\to\mathcal{B}_{\rho}(\delta)$ is contractive. Then the Banach fixed-point theorem yields that there is a unique element $v\in\mathcal{B}_{\rho}(\delta)$ such that $T(v)=v$. Thus we obtain a solution $(v,\nabla q)\in \big(L^2_\sigma(\Omega) \cap W^{1,2}_0(\Omega)^2 \cap W^{2,2}_{{\rm loc}}(\overline{\Omega})^2\big)\times L^2_{{\rm loc}}(\overline{\Omega})^2$ of \eqref{eq.NS.tilde} with $v$ unique in $\mathcal{B}_{\rho}(\delta)$.

By the argument so far, for a given $f\in l^1\big(L^\infty_{2\rho-1}(\Omega)\big)^2$ satisfying \eqref {est4.proof.thm.main}, the pair  
\begin{align*}
u := \alpha U - \gamma W + v, \qquad
\nabla p := \nabla\Big(-\frac{|u|^2}2 + q\Big)
\end{align*}
is a solution of \eqref{eq.NS.intro} in $\big(\widehat{W}^{1,2}(\Omega) \cap W^{2,2}_{{\rm loc}}(\overline{\Omega}) \cap L^\infty_{1} (\Omega)\big)^2 \times L^2_{{\rm loc}} (\overline{\Omega})^2$ unique in the set
\begin{align*}
\{(u,\nabla p) ~|~ u = \alpha U - \gamma W + v, \ \ v \in \mathcal{B}_{\rho}(\delta) \}. 
\end{align*}
The asymptotics \eqref{est2.thm.main} can be checked easily. The proof of Theorem \ref{thm.main} is complete. 
\end{proofx}
%

\appendix

\section{Integral estimates}\label{app.est.int.}
We summarize the integral estimates used in the proof of Proposition \ref{prop.est1.linear.nonzero}. The proof only uses elementary calculations and thus we will just state the results. Recall that 
\begin{align*}
n_\gamma = \Big\{n^2 + \Big(\frac{\gamma}{2}\Big)^2\Big\}^{\frac12}, 
\qquad 
\xi_n 
= 
\frac{n_\gamma}{\sqrt{2}} 
\bigg[\Big\{1 + \Big(\frac{\alpha n}{n_\gamma^2}\Big)^2\Big\}^{\frac12} + 1 \bigg]^{\frac12}. 
\end{align*}
%

%
\begin{lemma}\label{lem.app.est.int.1}
For $\alpha\in\R$, we have the following. 
\begin{enumerate}[(1)]
\item\label{item1.lem.app.est.int.1} Let $2<\rho\le\gamma$. For $|n|\ge1$, 
\begin{align*}
r^{-\xi_n-\frac{\gamma}{2}}
\int_{1}^{r} s^{\xi_n + \frac{\gamma}2 -2\rho + 1} \dd s
&
\le
\min\Big\{\frac{1}{\rho-2}, 
\Big(\xi_n-\frac{\gamma}2\Big)^{-1}\Big\}
r^{-\rho}, \\
r^{\xi_n-\frac{\gamma}{2}}
\int_{r}^{\infty} s^{-\xi_n + \frac{\gamma}2 -2\rho + 1} \dd s
&\le
\Big(\xi_n-\frac{\gamma}2\Big)^{-1} 
r^{-\rho}. 
\end{align*}

\item\label{item2.lem.app.est.int.1} Let $2<\rho\le3$. 
\begin{enumerate}[(i)]
\item If $2<\rho<3$, for $|n|\ge1$,
\begin{align*}
r^{-|n|-1}
\int_{1}^{r} s^{|n|-\rho+1} \dd s
&\le
\frac{1}{|n|-\rho+2} r^{-\rho+1}.
\end{align*}

\item If $\rho=3$, for $|n|=1$,
\begin{align*}
r^{-|n|-1}
\int_{1}^{r} s^{|n|-\rho+1} \dd s
&=
r^{-2} \log r,
\end{align*}
and for $|n|>1$,
\begin{align*}
r^{-|n|-1}
\int_{1}^{r} s^{|n|-\rho+1} \dd s
&\le
\frac{1}{|n|-1} r^{-2}.
\end{align*}

\item For $|n|\ge1$,
\begin{align*}
r^{|n|-1} \int_{r}^{\infty} s^{-|n|-\rho+1} \dd s
&\le \frac{1}{|n|+\rho-2} r^{-\rho+1}.
\end{align*}
\end{enumerate}
\end{enumerate}
\end{lemma}
%

%
\begin{lemma}\label{lem.app.est.int.2}
Under the assumption in Proposition \ref{prop.est1.linear.nonzero}, 
\begin{align*}
r^{-\xi_n-\frac{\gamma}{2}}
\int_{1}^{r} s^{\xi_n + \frac{\gamma}2} |f_{n}|(s) \dd s
&
\le
\Big(\xi_n-\frac{\gamma}2\Big)^{-\frac12} 
\bigg(\int_{1}^{r} s^4 |f_n|^2(s) s\dd s\bigg)^{\frac12} r^{-2}, \\
r^{\xi_n-\frac{\gamma}{2}}
\int_{r}^{\infty} s^{-\xi_n + \frac{\gamma}2} 
|f_{n}|(s)
\dd s 
&\le
\Big(\xi_n-\frac{\gamma}2\Big)^{-\frac12} 
\bigg(\int_{r}^{\infty} s^4 |f_n|^2(s) s\dd s\bigg)^{\frac12} r^{-2}. 
\end{align*}
\end{lemma}
%

%
\subsection*{Acknowledgements}
The author is partially supported by JSPS KAKENHI Grant Number JP 20K14345. 
%

%
\bibliography{Ref}
\bibliographystyle{plain}
%

\medskip

\begin{flushleft}
M. Higaki\\
Department of Mathematics, 
Graduate School of Science, 
Kobe University, 
1-1 Rokkodai, 
Nada-ku, 
Kobe 657-8501, 
Japan.

Email: higaki@math.kobe-u.ac.jp
\end{flushleft}

\medskip

\noindent \today

\end{document}